\documentclass[12pt]{article}
\usepackage[margin=1in]{geometry} 
\usepackage{amsmath,amsfonts,amsthm,amssymb}

\usepackage{tikz,amstext}
\usepackage{xcolor}
\usepackage{hyperref}

\newtheorem{proposition}{Proposition}
\newtheorem{theorem}{Theorem}
\newtheorem{lemma}{Lemma}
\newtheorem{corollary}{Corollary}

\newtheorem{remark}{Remark}

\title{Laplacian Energies of Vertices}
\author{J. Guerrero}
\date{\today}                 

\begin{document}
\maketitle

\begin{abstract}
	In this work, we define the Laplacian and Normalized Laplacian energies of vertices in a graph, we derive some of its properties and relate them to combinatorial, spectral and geometric quantities of the graph.
\end{abstract}

\section{Introduction}
	Nowadays, there are many notions of graph energies defined in the literature (see \cite{GutmanFurtula19} for a review of them). In the present work we will focus our interest in the \textit{Laplacian} (\cite{GutmanZhou06}) and \textit{normalized Laplacian} (\cite{CaversFallatKirkland10}) energies.

	We denote a graph by $G=(V,E)$, with $n$ vertices and $m$ edges, $V=\{v_1,v_2,\cdots,v_n\}$ is the vertex set, and $E$ is the set of edges. The \textit{degree} of a vertex $v_i$ is the number of vertices to which it is connected, and is denoted by $d_i$. If there is an edge between two vertices $v,w\in V$ we write $v\sim w.$  We will assume that $G$ is simple and connected.
	
	The \textit{$M-$energy} of a graph $G$ was introduced in \cite{CaversFallatKirkland10}, where $M$ is a matrix associated to $G$. The $M$-energy is defined by
	\[M\mathcal{E}(G):= \sum_{i=1}^{n} \bigg| \lambda_i(M) - \frac{tr(M)}{n} \bigg|= tr \left( \bigg|M- \frac{tr(M)}{n} I\bigg|\right),\]
	where $\lambda_1(M)\leq \lambda_2(M)\leq \cdots \leq \lambda_{n}(M)$ denote the eigenvalues of $M$, and the \textit{absolute value} of a matrix $B$ is defined as $|B|:= (BB^\ast)^{1/2}.$ The definition of $M-$energy is based on the work of Gutman, who in 1978 defined the \textit{energy} of a graph as $\mathcal{E}(G):=tr(|A|)$ (\cite{Gutman78}), where $A$ is the \textit{adjacency matrix} of $G$, defined by
	\[A_{ij} = \begin{cases}
	1 &\text{ if } v_i \sim v_j,\\
	0 & \text{ otherwise.}
	\end{cases}\]
	
	With the above definitioin in hand, the relevant energies for the present work, as their name suggest, are defined using the \textit{Laplacian} $L,$ and \textit{normalized Laplacian} $\mathcal{L},$ matrices:
	\[L := D- A  \text{ and } \mathcal{L} := I-D^{-1/2}AD^{-1/2},
	\]
	where $D$ is the diagonal \textit{degree matrix}, $D_{ii} =d_i.$
	
	Recently, the \textit{energy} of a vertex was defined in \cite{ArizmendiEtAl18}. {It has been not only a valuable tool to improve earlier known bounds on the energy of a graph, but also has shone a light on the behavior of the energy itself.} Given a vertex $v_i \in V,$ its energy is defined as
	\[\mathcal{E}_G(v_i) := |A|_{ii}.\]
	Behind the definition, lies the fact that the trace of a matrix can be decomposed as the sum of the positive linear functionals $\phi_i:M_n(\mathbb{R}) \to \mathbb{R}$ defined by $\phi_i(M) = M_{ii}.$ The positivity of the linear functionals imply, in particular, the Cauchy-Schwarz inequality
	\begin{equation}
		\label{eq:CauchySchwarz}
		|\phi_i(AB)| \leq \phi_i(AA^*)^{1/2}\phi_i(BB^*)^{1/2}.
	\end{equation}
	
	It is immediate from the definition, that $\mathcal{E}(G) = \mathcal{E}_G(v_1)+\mathcal{E}_G(v_2)+\cdots+ \mathcal{E}_G(v_n),$ and so, as it is said in \cite{ArizmendiEtAl18}, the energy of a vertex might be understood as the contribution of each vertex to the total energy of the graph.
	
	The purpose of this work is to define the energy of a vertex associated to the matrices $L$ and $\mathcal{L}$, investigate their properties and relate them to spectral, combinatorial and geometric quantities of the graph. While some of the theorems follow similar developments as the one given for the energy of a vertex \cite{ArizmendiEtAl18, ArizmendiLunaRamirez19, ArizmendiArizmendi21}, by the geometric properties of $L$ and $\mathcal{L}$, different results and new insights are obtained in this paper for their vertex energies.
	
	We present different methods to compute these energies of a vertex: in terms of the eigendecomposition of a matrix (Proposition \ref{prop:EnergySpectral}), in terms of the distribution of the graph with respect to the vertex (Remark \ref{remark:NonCommutativeDistribution}), and in terms of the characteristic polynomial of the matrix (Theorem \ref{thm:CoulsonIntegralFormula}); we show that the laplacian vertex energies satisfy some classical inequalities (Theorems \ref{thm:McClellandLaplacian}, \ref{thm:LowerLaplacian}, \ref{thm:NLEUpperBound}, \ref{thm:NLELowerBound}), and it turns out that they add up to improvements on the bounds for the energy of the graph (Remarks \ref{remark:improvement1}, \ref{remark:improvement2}); we also relate the normalized laplacian vertex energies to some geometric quantities of the graph (Theorems \ref{thm:vertexEnergyAndCheegerConstant}, \ref{thm:vertexEnergyAndCurvature}); finally we show some computations of the vertex energies for two families of graphs.
	
	The content of this paper is organized as follows: in Section \ref{section:MEnergy} a framework is developed which allows us to prove general facts about the energy of vertices; in Section \ref{section:LE} the Laplacian energy of vertices is defined and some inequalities are proved; Section \ref{section:NLE} is about the normalized Laplacian energy of vertices, it is divided in two parts: the first relate it to the degrees of the graph and the second to the Cheeger and dual Cheeger constants and to the Ollivier-Ricci curvature; finally, Section \ref{section:Examples} is devoted to give formulas for the energies of vertices of the Star and Path graphs.

\section{$M$-energy of a Vertex}
\label{section:MEnergy}

	Some of the properties of the vertex energies associated to $L$ and $\mathcal{L}$ will also have analogues for other matrices as well. In this section we will develop a simple frame, but general enough to study the energy of a vertex given by a real symmetric matrix associated to a graph $G$. We follow the lines developed in \cite{ArizmendiEtAl18} and \cite{ArizmendiLunaRamirez19}.
	
	The quantity of interest throughout this section is 
	\[M\mathcal{E}_G(v_i) := \phi_i\left(\bigg|M- \frac{tr(M)}{n} I\bigg|\right) = \bigg|M- \frac{tr(M)}{n} I\bigg|_{ii} ,\]
	where $M$ is a symmetric matrix of size $n \times n,$ associated to $G$.
	The following result shows how to compute it in terms of the spectral decomposition of $M$.
	
	\begin{proposition}
		\label{prop:EnergySpectral}
		Let $M$ be a real symmetric matrix and $M=U\Lambda(M) U^T$ be its spectral decomposition, where $U$ is orthonormal and $\Lambda(M) = diag(\lambda_1(M),\lambda_2(M),\cdots, \lambda_{n}(M)),$ then
		\[M\mathcal{E}_G(v_i) = \sum_{j=1}^{n} U_{ij}^2\bigg| \lambda_j(M) - \frac{tr(M)}{n} \bigg|.\]
	\end{proposition}
	\begin{proof}
		The spectral decomposition of $M$ allow us to write 
		\[\bigg| M - \frac{tr(M)}{n}I \bigg| = U\bigg| \Lambda(M) - \frac{tr(M)}{n}I \bigg| U^T ,\]
		where $|\Lambda(M)- (tr(M)/n)I|_{jk} = |\lambda_j(M) - (tr(M)/n)|\delta_{jk} .$ Finally we have
		\[\bigg| M - \frac{tr(M)}{n}I \bigg|_{ii} =	\sum_{j=1}^{n}\sum_{k=1}^{n}U_{ij}|\Lambda(M)- (tr(M)/n)I|_{jk}U_{ik}=\sum_{j=1}^{n} U_{ij}^2\bigg| \lambda_j(M) - \frac{tr(M)}{n} \bigg|.	\qedhere\]
	\end{proof}
	
	\begin{remark}
		\label{remark:NonCommutativeDistribution}
		The weights involved in the sum in the previous result satisfy:
		\[\sum_{i=1}^{n}U_{ij}^2 = 1=\sum_{j=1}^{n}U_{ij}^2.\]
		In the context of non-commutative probability one would say that the distribution of the matrix $M$ with respect to the functional $\phi_i$ is the probability measure
		\[\mu(M) = \sum_{j=1}^{n}U_{ij}^2 \delta_{\lambda_j(M)}. \]
		For an introduction to non-commutative probability theory see \cite{NicaSpeicher06}.
	\end{remark}

	The following identity is known as the Coulson Integral Formula, it provide us with another method for computing the energy of a graph:
	\[\mathcal{E}(G) = \frac{1}{\pi}\int_{\mathbb{R}} \left[n-\frac{ix\phi'(A;ix)}{\phi(A;ix)}\right]dx,\]
	where $\phi(A;x)$ is the characteristic polynomial of $A$. It was first stated in \cite{Coulson40}, in the context of organic molecules, and in \cite{Gutman78} for the graph energy. Later, in  \cite{PenaRada16}, a variant of this formula was shown to be true for the (outer) energy associated to a symmetric matrix. The Coulson Integral Formula has the advantage of depending only on the characteristic polynomial of $A$, without the need of knowing its whole spectrum.
	
	A version of the Coulson Integral Formula for the energy of a vertex was stated in \cite{ArizmendiLunaRamirez19}. We show now that a similar formula holds for the energy $M\mathcal{E}(v_i).$ First we state the following lemma.
	
	\begin{lemma}
		\label{lemma:PsiIdentity}
		Let $M = U\Lambda(M) U^T$ be the spectral decomposition of $M$, then
		\[\sum_{j=1}^{n}\frac{U_{ij}^2}{z-\left(\lambda_j(M) - \frac{tr(M)}{n}\right)} 
		= \frac{det\left[\left(z-\frac{tr(M)}{n}\right)I -\widetilde{ M}_{ii}\right]}{det\left[\left(z-\frac{tr(M)}{n}\right)I - M\right]} ,\]
		where $\widetilde{ M}_{ii}$ is the matrix obtained from $M$ by eliminating its $i$-th column and its $i$-th row.
	\end{lemma}  
	\begin{proof}
		We define $\Psi_i(z) = \sum_{k=0}^{\infty} [(M-(tr(M)/n)I)^k]_{ii} z^{-k-1}.$ On one hand, we arrive at
		\[\Psi_i(z) = \sum_{j=1}^{n} \frac{U_{ij}^2}{z-\left(\lambda_j(M) - \frac{tr(M)}{n}\right)}\]
		by noticing that $[(M-(tr(M)/n)I)^k]_{ii} = \sum_{j=1}^{n} U_{ij}^2\left(\lambda_j(M) - \frac{tr(M)}{n}\right)^k. $ On the other, we can also write
		\begin{align*}
			\Psi_i(z) &= \frac{1}{z} \left[\sum_{k=0}^\infty \left(\frac{M-(tr(M)/n)I}{z}\right)^k\right]_{ii}\\
			&= \left[\left(\left(z-\frac{tr(M)}{n}\right)I - M\right)^{-1} \right]_{ii} \\
			&=  \frac{det\left[\left(z-\frac{tr(M)}{n}\right)I -\widetilde{ M}_{ii}\right]}{det\left[\left(z-\frac{tr(M)}{n}\right)I - M\right]} ,
		\end{align*}
		where in the last equality we used the expression of the inverse of a matrix in terms of its adjugate. This proves the lemma.
	\end{proof}

	To prove the integral formula for the energy of a vertex, we follow a similar idea as in the proof of Theorem 1 in \cite{Gutman78} which avoids the use of contour integration.
	
	\begin{theorem}[Coulson Integral Formula for a Vertex]
		\label{thm:CoulsonIntegralFormula}
		Let $G$ be a graph and $v_i$ a vertex in it. Then
		\[M\mathcal{E}_G(v_i) = \frac{1}{\pi} \int_{\mathbb{R}} \left(1 - ix \frac{det\left[\left(ix-\frac{tr(M)}{n}\right)I -\widetilde{ M}_{ii}\right]}{det\left[\left(ix-\frac{tr(M)}{n}\right)I - M\right]} \right) dx\]
	\end{theorem}
	\begin{proof}
		Given that $\int xt(t^2+x^2)^{-1}dx =0,$ we can write
		\begin{align*}
		|t| &= \frac{1}{\pi}\int_{\mathbb{R}} \frac{t^2}{t^2+x^2}dx + i \frac{1}{\pi}\int_{\mathbb{R}} \frac{xt}{t^2 + x^2}dx = \frac{1}{\pi} \int_{\mathbb{R}} \left(1 - \frac{ix}{ix-t}\right) dx.
		\end{align*}
		Making the substitution $t=\lambda_j(M)- tr(M)n^{-1},$ multiplying by $U_{ij}^2$ and taking the sum over $j$, we obtain by Proposition \ref{prop:EnergySpectral}
		\[M\mathcal{E}_G(v_i)= \sum_{j=1}^{n} U_{ij}^2 \bigg|\lambda_i(M) - \frac{tr(M)}{n}\bigg| = \frac{1}{\pi} \int_{\mathbb{R}} \left(1 - ix \sum_{j=1}^{n}\frac{U_{ij}^2}{ix-\left(\lambda_i(M) - \frac{tr(M)}{n}\right)}\right) dx.\]
		Finally Lemma \ref{lemma:PsiIdentity} allow us to replace the sum in the integrand which conclude the proof.
	\end{proof}
	
	We finish this section by stating a lemma that would later relate the Laplacian and normalized Laplacian energies with the Randi\'c index. The proof is a clever usage of the Cauchy-Schwarz inequality as it is done in \cite{ArizmendiArizmendi21}.
	
	\begin{lemma}
		\label{lemma:CS-AGMean}
		Let $G$ be a connected graph and $v_i, v_j$ adjacent vertices. Then 
		\begin{enumerate}
			\item[(i)] $M\mathcal{E}_G(v_i) M\mathcal{E}_G(v_j)\geq [M- (tr(M)/n)I]_{ij}^2,$
			\item[(ii)] $M\mathcal{E}_G(v_i) +M\mathcal{E}_G(v_j) \geq 2\bigg|\left[M-\frac{tr(M)}{n}I\right]_{ij}\bigg|.$
		\end{enumerate}
	\end{lemma}
	\begin{proof}
		As before, we use the spectral decomposition of $M$ to write $M=U\Lambda(M)U^T$. In order to prove the first inequality, we define the vectors:
		\begin{align*}
			w_i &= (U_{i1} \sqrt{\lambda_1(M)-(tr(M)/n)},\cdots,U_{in} \sqrt{\lambda_n(M)-(tr(M)/n)}),\\
			w_j &= (U_{j1} \sigma_1\sqrt{\lambda_1(M)-(tr(M)/n)},\cdots,U_{jn}\sigma_n \sqrt{\lambda_n(M)-(tr(M)/n)}),
		\end{align*}
		where $\sigma_k=sign(\lambda_k(M)-(tr(M)/n)).$ By Proposition \ref{prop:EnergySpectral}, we have
		\[\|w_i\|^2 = M\mathcal{E}(v_i),\; \|w_j\|^2 = M\mathcal{E}(v_j).\]
		Also, we can compute the inner product of $w_i$ and $w_j$:
		\[\langle w_i, w_j \rangle = \sum_{k=1}^n U_{ik}U_{jk}\left(\lambda_k(M)-\frac{tr(M)}{n}\right)=\left[M-\frac{tr(M)}{n}I\right]_{ij},\]
		and we obtain \textit{(i)} using the Cauchy-Schwarz inequality.
		
		The second part of the lemma is just an application of the inequality of arithmetic and geometric means and the previous inequality:
		\[M\mathcal{E}(v_i) +M\mathcal{E}(v_j)\geq 2 \sqrt{M\mathcal{E}(v_i)M\mathcal{E}(v_j)} \geq 2\bigg|\left[M-\frac{tr(M)}{n}I\right]_{ij}\bigg|.\qedhere\] 	
	\end{proof}

\section{Laplacian Energy of a Vertex}
\label{section:LE}
In this section we study the energy of vertices obtained by using the Laplacian Matrix $L$. Given a graph $G=(V,E),$ we define the \textit{Laplacian energy} of a vertex $v_i\in V$ by

\[L\mathcal{E}_G(v_i) := \phi_i \left(\bigg|L-\frac{2m}{n}Id\bigg|\right),
\]
where we used the fact that $tr(L)=\sum_{i=1}^n d_i = 2m.$

The \textit{Laplacian energy} of a Graph was defined in \cite{GutmanZhou06} as $L\mathcal{E}(G) = tr \left(\bigg|L-\frac{2m}{n}Id\bigg|\right)$, its properties had been investigated also in \cite{DasMojallal2014}, and in \cite{PirzadaGanie15} among others. It is immediate that 
\begin{equation}
\label{eq:sumOfLEnergies}
L\mathcal{E}(G) = \sum_{i=1}^{n}L\mathcal{E}_G(v_i).
\end{equation}

In \cite{GutmanZhou06}, it was also proved that if $G$ is a $d$-regular graph, the Laplacian energy and the energy would coincide. We have an analogous result for the energies of vertices.

\begin{lemma}
	\label{lemma:regularGraphs}
	If $G$ is a regular graph with degree $d$, then $L\mathcal{E}_G(v_i)= \mathcal{E}_G(v_i)$ for every vertex. 
\end{lemma}
\begin{proof}
	If $G$ is $d$-regular, then $2m/n=d$ and the matrices $(2m/n)Id$ and $D$ are equal. So $|L-(2m/n)Id| = |D-A-(2m/n)Id| =|-A|=|A|.$ This way:
	\[L\mathcal{E}_G(v_i)=\phi_i(|L-(2m/n)Id|) = \phi_i(|A|)= \mathcal{E}_G(v_i).\qedhere\]
\end{proof}

We now proceed to give an upper and a lower bound on the Laplacian energy of a vertex. The former is analogous to the one given in Theorem 2 of \cite{GutmanZhou06}, which in turn is analogous to the McClelland inequality \cite{McClelland71}.

\begin{theorem}
	\label{thm:McClellandLaplacian}
	For a graph $G$ and a vertex $v_i \in G$ with degree $d_i$, we have
	\[L\mathcal{E}_G(v_i) \leq \sqrt{d_i + \left(\frac{2m}{n} - d_i\right)^2}.\]
	The equality holds if and only if $n\leq2.$
\end{theorem}
\begin{proof}
	Let $L=U\Lambda U^T$ be the spectral decomposition of $L$; we consider the vectors:
	\[	a = (|U_{i1}|, |U_{i2}|, \cdots, |U_{in}|),\]
	\[b = (|U_{i1}||\lambda_1(L) - 2m/n|, |U_{i2}||\lambda_2(L) - 2m/n|, \cdots, |U_{in}||\lambda_n(L) - 2m/n|).\]
	If we use the Cauchy-Schwarz inequality with these vectors we get:
	\[
	LE(v_i) = \sum_{j=1}^{n} U_{ij}^2 |\lambda_j(L) - 2m/n| = a\cdot b \leq \sqrt{\sum_{j=1}^n |U_{ij}|^2 \sum_{j=1}^n |U_{ij}|^2 |\lambda_j(L) - 2m/n|^2},
	\]
	where the first sum inside the sqare root is equal to $1$, and the second sum is equal to $[(L- (2m/n)I)^2]_{ii},$ which can be computed as follows:
	\begin{align*}
		\left[\left(L-\left(\frac{2m}{n}\right)Id\right)^2\right]_{ii} &= \left[\left(D-A-\left(\frac{2m}{n}\right)Id\right)^2\right]_{ii}  \\
		&=\left[ \left(\frac{2m}{n}\right)^2Id -2\left(\frac{2m}{n}\right)D+D^2+A^2 +2\left(\frac{2m}{n}\right)A-(DA+AD)\right]_{ii}\\
		&= \left(\frac{2m}{n}\right)^2 -2 \left(\frac{2m}{n}\right) d_i + d_i^2 + d_i
		= \left(\frac{2m}{n} - d_i\right)^2 +d_i, 
	\end{align*}
	which shows the inequality. 
	
	The equality holds if and only if $a$ and $b$ are parallel, which can only be true if one have
	\[|\lambda_j(L) - 2m/n| = |\lambda_k(L) - 2m/n|\]
	for every $j$ and $k$ such that $U_{ij} \neq 0$ and $U_{ik}\neq 0$.	Since $G$ is connected $\lambda_1(L)=0,$ furthermore we can take as an the first column of $U$ the egienvector $(1/\sqrt{n},1/\sqrt{n},\cdots, 1/\sqrt{n})$ associated to that eigenvalue, so $U_{i1} \neq 0$ and we arrive at
	\[|\lambda_k(L) - 2m/n| = |\lambda_1(L) - 2m/n|= \frac{2m}{n}, \;  \text{ thus }\lambda_k(L) = \frac{4m}{n}\]
	for all $n\geq k > 1$ such that $U_{ik} \neq 0.$ 
	Those eigenvalues allow us to write the first three moments of $L$ with respect to $\varphi_i$:
	
	\begin{align*}	
	d_i = \varphi_i(L) =& \sum_{k=2}^nU_{ik}^2 \left(\frac{4m}{n}\right), \\
	d_i^2 + d_i = \varphi_i(L^2) =& \sum_{k=2}^nU_{ik}^2 \left(\frac{4m}{n}\right)^2\\
	d_i^3 +2d_i^2 + \sum_{k\sim i}d_k - \Delta_i = \varphi_i(L^3) =& \sum_{k=2}^nU_{ik}^2 \left(\frac{4m}{n}\right)^3
	\end{align*}
	
	where $\Delta_i$ is the number of triangles in which $v_i$ lies. Furthermore, from the orthogonality of $U$ we can also state that
	\begin{equation}
		\label{eq:orthogonalityOfU}
		\sum_{k=2}^n U_{ik}^2 + 1/n = 1, \text{ for every } i\in \{1,\cdots, n\}.
	\end{equation} 
	All of the above equations imply the following equalities:
	\[d_i = n-1,\; m = n^2 / 4, \; \Delta_i = (n-2)^2,\]
	in particular, the first one tells us that $v_i$ is connected to all the other vertices in the graph, thus for every other edge there would be a triangle containing $v_i$, meaning that
	\[m = d_i + \Delta_i,\]
	allowing us to conclude that $n=2$, and since the graph is connected, we would have $G= K_2.$
\end{proof}
	\begin{remark}
		\label{remark:improvement1}
		The bound obtained in the previous theorem gives an improvement over inequality 10 of Theorem 2 in \cite{GutmanZhou06}, that is to say
		\[ L\mathcal{E}(G) = \sum_{i=1}^nL\mathcal{E}_G(v_i) \leq \sum_{i=0}^n \sqrt{d_i + \left(d_i-´\frac{2m}{n}\right)^2} \leq \sqrt{n\sum_{i=1}^n d_i+\left(d_i-´\frac{2m}{n}\right)^2} .\]
	\end{remark}
	
	Now we give a lower bound in terms of the degrees of the graph, similar to Therorem 3.3 in \cite{ArizmendiEtAl18}. We remind that $\lambda_n(L)$ denote the largest eigenvalue of $L$. It was proved in \cite{Kelmans67} that if $G$ is a simple graph, then $\lambda_n(L) \leq n,$ with equality if and only if the complement of $G$ is not connected.
	
	\begin{theorem}
		\label{thm:LowerLaplacian}
		Let $G$ be a simple graph, $v_i \in V$ of degree $d_i,$ and $n' = \max\{\frac{2m}{n}, n-\frac{2m}{n}\}$, then 
		\[L\mathcal{E}_G(v_i) \geq \frac{\left(\frac{2m}{n} - d_i\right)^2 + d_i}{n'}.\]
		Equality hods {if and only if the spectrum of $G$ takes the values $\{0, 2m/n, n\}.$}
	\end{theorem}
	\begin{proof}
		Notice that $|x|\geq x^2$ for $x\in [-1,1].$ Since $\lambda_n(L) \leq n,$ then $|\lambda_i(L) - \frac{2m}{n}|\leq n'$ for all $i.$  We conclude by
		\[L\mathcal{E}_G(v_i) = \sum_{j=1}^{n} U_{ij}^2 \Big|\lambda_j(L) - \frac{2m}{n}\Big| \geq \sum_{j=1}^{n} U_{ij}^2 \frac{(\lambda_j(L) - \frac{2m}{n})^2}{n'} = \frac{\phi_i\left([L-\frac{2m}{n}I]^2\right)}{n'}
		= \frac{\left(\frac{2m}{n} - d_i\right)^2 + d_i}{n'}.
		\]
		Equality holds if and only if $|\lambda_i(L) - \frac{2m}{n}|= n'$ for every $i$, which implies that the spectrum of $L$ is contained in $\{0, 2m/n, n\}.$				
	\end{proof}
	\begin{remark} We can say more if the last inequality is attained. First note that for a complete graph does not hold the equality, because $L\mathcal{E}(v) = \mathcal{E}(v) = 2(n-1)n^{-1} > 1$, for every vertex $v$. This implies that for if a graph attains the equality, then $\lambda_2(L)$ must be equal to $2m/n.$ If $\lambda_{n}(L) = 2m/n,$ then all the nontrivial eigenvalues have to be equal to this value, which would be also equal to $(n-1)n^{-1}d_{min}$ and to $(n-1)n^{-1}d_{max}$ this imply that the graph must be regular; in this case we would have
		\[d = \frac{2m}{n} = \frac{n-1}{n}d,\]
		which leads to a contradiction. We conclude $\lambda_{n}(L)=n.$ That is, the spectrum of $G$ must take the three values $\{0, 2m/n, n\}.$
	\end{remark}
	
	We finish this section, showing that lemma \ref{lemma:CS-AGMean} has the following result as corollary. For the first statement, one only need to notice that if $v_i$ and $v_j$ are adjacent vertices, then $\left[L-I\right]_{ij} =-1.$
	
	\begin{lemma}
		\label{lemma:CS-AGMeanL}
		Let $G$ be a connected graph and $v_i, v_j$ adjacent vertices. Then 
		\begin{itemize}
			\item $L\mathcal{E}_G(v_i) L\mathcal{E}_G(v_j)\geq 1,$
			\item $L\mathcal{E}_G(v_i) +L\mathcal{E}_G(v_j) \geq 2.$
		\end{itemize}
	\end{lemma}
	
	The second result relates the Laplacian energy of a graph with a version of the Randi\'c index defined as \[R_{-1/2}(G) := \sum_{\{v_i,v_j\}\in E} d_i^{-1/2}d_j^{-1/2}.\]
	\begin{theorem}
		Let $G$ be a connected graph, then 
		\[L\mathcal{E}(G) \geq 2R_{-1/2}(G),\]
	\end{theorem}
	\begin{proof}
		This result is again an application of the inequality of arithmetic and geometric means:
		\begin{align*}
			L\mathcal{E}(G)&= \sum_{i=1}^{n} L\mathcal{E}_G(v_i) \\
			&= \sum_{\{v_i,v_j\}\in E} \frac{L\mathcal{E}_G(v_i)}{d_i} +\frac{L\mathcal{E}_G(v_j)}{d_j} \\
			&\geq 2  \sum_{\{v_i,v_j\}\in E} \sqrt{\frac{L\mathcal{E}_G(v_i)L\mathcal{E}_G(v_j)}{d_i d_j}} \\
			&\geq 2  \sum_{\{v_i,v_j\}\in E}\frac{1}{ \sqrt{{d_i d_j}}},
		\end{align*}
		where we have used lemma \ref{lemma:CS-AGMeanL} in the last inequality. 
	\end{proof}
	
	\begin{remark}
		In \cite{ArizmendiArizmendi21}, the energy of an edge $e=\{v_i,v_j\}\in E$ is defined as the sum
		\[\mathcal{E}_G(e)=\frac{\mathcal{E}_G(v_i)}{d_i} +\frac{\mathcal{E}_G(v_j)}{d_j}.\]
		This definition satisfies $\mathcal{E}(G) = \sum_{e\in E}\mathcal{E}_G(e) .$ 
		
		In the previous proof, we decomposed the Laplacian energy of $G$ in a similar fashion, which can serve to define a Laplacian energy of an edge.
	\end{remark}

\section{Normalized Laplacian Energy of a Vertex}
\label{section:NLE}

	We now define and study the energy of a vertex associated to the normalized Laplacian $\mathcal{L}.$ For a connected simple graph $G=(V,E)$ we define the \textit{normalized Laplacian energy} of a vertex $v_i \in V$ as 
	\[\mathcal{LE}_G(v_i) := \phi_i(|\mathcal{L}-I|).\]
	Note that this definition uses the fact that $tr(\mathcal{L})=n .$ 
	
	It is known that the spectrum of $\mathcal{L}$  satisfies
	\[0=\lambda_1(\mathcal{L})\leq \lambda_2(\mathcal{L})\leq \cdots \leq \lambda_n(\mathcal{L})\leq 2,\]
	see for example \cite{Chung97}. This fact, together with Proposition \ref{prop:EnergySpectral} imply 
	\[\mathcal{LE}_G(v_i) \in [0,1].\]
	
	The \textit{normalized Laplacian energy} of a graph was defined in \cite{CaversFallatKirkland10} as $\mathcal{LE}= tr(|\mathcal{L}-I|)$, in there some of its properties were investigated, including its relationship with the Randi\'c index 
	\[R_{-1}(G):=\sum_{\{v_i,v_j\}\in E} d_i^{-1} d_j^{-1}.\]
	
	It is immediate the following identity
	\[\mathcal{LE}(G)=\sum_{i=1}^{n} \mathcal{LE}_G(v_i).\]
	As it was said in section \ref{section:LE}, the energy and Laplacian energy coincide for regular graphs, and in Lemma \ref{lemma:regularGraphs} an analogous result for the vertex energy was stated. We present the corresponding result for the normalized Laplacian energy.
	
	\begin{lemma}
		\label{lemma:regularGraphsNormalized}
		If $G$ is a regular graph with degree $d$, then $\mathcal{E}_G(v_i)=d\mathcal{LE}_G(v_i) $ for every vertex. 
	\end{lemma}
	\begin{proof}
		We first note that, for a $d$-regular graph:
		\[|\mathcal{L}-I_n| = |I_n -D_G^{-1/2}A_GD_G^{-1/2}-I_n| = |D_G^{-1/2}A_GD_G^{-1/2}|.\]
		Furthermore if $G$ is regular, then $D=dI_n$ and $|\mathcal{L}-I_n| = d^{-1}|A|.$ This way:
		\[\mathcal{LE}_G(v_i)=\phi_i(|L-I_n|) = \phi_i(d^{-1}|A|)= d^{-1}\mathcal{E}_G(v_i). \qedhere\] 
	\end{proof}

	\subsection{Bounds with the Vertex Degrees}
	
		We now give an upper and a lower bound of the normalized Laplacian energy in terms of the degrees of the vertices.
	
		\begin{theorem}
			\label{thm:NLEUpperBound}
			If $G$ is a connected graph, then 
			\[\mathcal{LE}_G(v_i) \leq \sqrt{\frac{1}{d_i} \sum_{w\in V:w \sim v_i} \frac{1}{d_w}}.\]
			The equality holds if and only if $G$ is a star and $v_i$ is its center.
		\end{theorem}
		\begin{proof}
			Let $\mathcal{L}=W\Lambda(\mathcal{L}) W^T$ be the spectral decomposition of $L$; we consider the vectors:
			\[	a = (|W_{i1}|, |W_{i2}|, \cdots, |W_{in}|),\]
			\[b = (|W_{i1}||\lambda_1(\mathcal{L}) - 1|, |W_{i2}||\lambda_2(\mathcal{L}) - 1|, \cdots, |W_{in}||\lambda_n(\mathcal{L}) - 1|).\]
			If we use the Cauchy-Schwarz inequality with these vectors we get:
			\[
			\mathcal{LE}(v_i) = \sum_{j=1}^{n} W_{ij}^2 |\lambda_j(\mathcal{L}) - 1| = a\cdot b \leq \sqrt{\sum_{j=1}^n |W_{ij}|^2 \sum_{j=1}^n |W_{ij}|^2 |\lambda_j(\mathcal{L}) - 1|^2},
			\]
			where the first sum inside the sqare root is equal to $1$, and the second sum is equal to \[[|\mathcal{L}-I|^2]_{ii} = [(\mathcal{L}-I)^2]_{ii}= [D^{-1/2}AD^{-1}AD^{-1/2}]_{ii},\]
			which can be computed as
			\begin{align*}
				[|\mathcal{L}-I|^2]_{ii} &=
				\sum_{j,k,l,m=1}^n  [D^{-1/2}]_{ij}A_{jk}[D^{-1}]_{kl}A_{lm}[D^{-1/2}]_{mi} \\
				&=d_i^{-1}\sum_{k=1}^n  A_{ik}[d_k^{-1}]A_{ki}\\
				&= \frac{1}{d_i} \sum_{k:v_k \sim v_i} \frac{1}{d_k},
			\end{align*}
			and the inequality is proved.
			
			For the equality to hold, it must be the case that $a$ and $b$ are parallel, which means that $|\lambda_j(\mathcal{L}) - 1| = |\lambda_k(\mathcal{L}) - 1|$ for every $j,k$ such that $W_{ik} \neq 0  \neq W_{ij}.$ This, together with the fact that $0$ is always an eigenvalue of $\mathcal{L}$ with an eigenvector such that $W_{i1} \neq 0$ (and also with multiplicity one, because $G$ is connected),   have as a consequence the result that the eigenvalues of $\mathcal{L},$ such that $W_{ik} \neq 0,$ are contained in $\{0, 2\}$. 
			
			The fact that we can choose $W_{i1} \neq 0$ for every $i$, follows because one eigenvector associated to the eigenvalue $0$ is given by $D^{1/2}\textbf{1}$ (where $\textbf{1}$ is the vector with all its entries equal to 1) and so we can choose as the first column of $W$ the vector:
			\[\left(\sqrt{\frac{d_1}{2m}}, \sqrt{\frac{d_2}{2m}}, \cdots, \sqrt{\frac{d_n}{2m}}\right).\]
			From the orthogonality of $W$, we can write
			\begin{equation}
				\label{eq:orthogonalityOfW}
				\sum_{k=2}^n W_{ik}^2 + \frac{d_i}{2m} = 1, \text{ for every } i \in \{1,\cdots,n\},
			\end{equation} 
			and using the first moment of $\mathcal{L}$ with respect to $\varphi_i$:
			\begin{equation}
				\label{eq:FirstMomentL}
				\varphi_i(\mathcal{L}) = 1 = 2\sum_{k=2}^n W_{ik}^2,
			\end{equation}
			This implies that $\sum_{k=2}^n W_{ik}^2 = 1/2$, and, together with the equation \eqref{eq:orthogonalityOfW}, imply $d_i = m$: that is to say $v_i$ is the center of a star.
		\end{proof}

		\begin{remark}
			\label{remark:improvement2}
			We can compare the previous bound with the inequality given in Lemma 1 in \cite{CaversFallatKirkland10} as follows:
			\[\mathcal{LE}(G) = \sum_{i=0}^n \mathcal{LE}_G(v_i)\leq \sum_{i=0}^{n} \sqrt{\frac{1}{d_i} \sum_{w:w \sim v_i} \frac{1}{d_w}} \leq \sqrt{2nR_{-1}(G)}.\]
		\end{remark}
		
		\begin{theorem}
			\label{thm:NLELowerBound}
			If $G$ is a connected graph, then 
			\[\mathcal{LE}_G(v_i) \geq \frac{1}{d_i} \sum_{w:w \sim v_i} \frac{1}{d_w}.\]
			 {The equality holds if and only if $G$ is a complete bipartite graph}.
		\end{theorem}
		\begin{proof}
			Given that $\lambda_i - 1 \in [-1,1]$, we have $(\lambda_i - 1)^2 \leq |\lambda_i - 1|,$ which implies
			\[\mathcal{LE}_G(v) = \sum_{i=0}^{n-1}  V_{ki}^2 |\lambda_i(\mathcal{L}) - 1| \geq \sum_{i=0}^{n-1}  V_{ki}^2 (\lambda_i(\mathcal{L}) - 1)^2 = \phi_i[(\mathcal{L}-I_n)^2]=\frac{1}{d_i} \sum_{w:w \sim v_i} \frac{1}{d_w}. \]
			{The equality happens if and only if $|\lambda_i-1|\in \{0,1\}$, that is, if and only if the spectrum of $\mathcal{L}$ takes values on $\{0,1,2\},$ which can only happen if and only if $G$ is a bipartite graph (Lemma 1.8 of \cite{Chung97}). If $G$ is a bipartite graph, then the spectrum of $\mathcal{L}$ is symmetric with respect to $1$, which together with the fact that $G$ is connected, imply that the multiplicity of $2$ must be $1$, and thus the multiplicity of $1$ must be $n-2$, which can only happen if and only if $G$ is a complete bipartite graph.}
		\end{proof}
		
		The two previous theorems give us as a corollary the following bounds. As we will see in section \ref{section:StarGraphs}, the star graphs attain these bounds.
		
		\begin{corollary}
			\label{cor:NLEDegreeBounds}
			For a connected graph $G=(V,E)$ and a vertex $v_i\in V$ we have
			\[\frac{1}{d_{max}} \leq \mathcal{LE}_G(v_i) \leq \frac{1}{\sqrt{d_{min}}}.\]
		\end{corollary}		
		
		Just as for the Laplacian energy, Lemma \ref{lemma:CS-AGMean} has a precise form for this energy. To state it we need to know that, if the vertices $v_i$ and $v_j$ are adjacent, then $[\mathcal{L}-I]_{ij} = -(d_i d_j)^{-1/2}.$
		
		\begin{lemma}
			\label{lemma:CS-AGMeanNL}
			Let $G$ be a connected graph and $v_i,v_j$ adjacent vertices. Then
			\begin{itemize}
				\item $\mathcal{LE}_G(v_i) \mathcal{LE}_G(v_j)\geq (d_i d_j)^{-1},$
				\item $\mathcal{LE}_G(v_i) +\mathcal{LE}_G(v_j) \geq 2 (d_id_j)^{-1/2}.$
			\end{itemize}
		\end{lemma}
	
		Now we are in position to give two proofs of the first inequality appearing in Lemma 1 of \cite{CaversFallatKirkland10}, relating the normalized Laplacian energy of $G$ and the Randi\'c index $R_{-1}.$
		
		\begin{theorem}[\cite{CaversFallatKirkland10}]
			Let $G$ be a connected graph, then 
			\[\mathcal{LE}(G) \geq 2R_{-1}(G),\]
		\end{theorem}
		\begin{proof}
			The first proof follows immediately from theorem \ref{thm:NLELowerBound} by taking the sum over all vertices:
			\[\sum_{v_i\in V} \mathcal{LE}_G(v_i) \geq \sum_{v_i\in V} \frac{1}{d_i} \sum_{w:w \sim v_i} \frac{1}{d_w} = 2 R_{-1(G)}.\]
			
			As the reader might infer, the second proof follows from an application of the inequality of arithmetic and geometric means:
			\begin{align*}
			\mathcal{LE}(G)&= \sum_{i=1}^{n} \mathcal{LE}_G(v_i) \\
			&= \sum_{\{v_i,v_j\}\in E} \frac{\mathcal{LE}_G(v_i)}{d_i} +\frac{\mathcal{LE}_G(v_j)}{d_j} \\
			&\geq 2  \sum_{\{v_i,v_j\}\in E} \sqrt{\frac{\mathcal{LE}_G(v_i)\mathcal{LE}_G(v_j)}{d_i d_j}} \\
			&\geq 2  \sum_{\{v_i,v_j\}\in E}\frac{1}{ {d_i d_j}},
			\end{align*}
			where we have used lemma \ref{lemma:CS-AGMeanNL} in the last inequality. 
		\end{proof}
	
		\begin{remark}[Conjecture]
			In Theorem 17 of \cite{CaversFallatKirkland10} it is proven that \[d_{min} \mathcal{LE}(G) \leq \mathcal{E}(G) \leq d_{max} \mathcal{LE}(G),\]
			where $d_{min}$ and $d_{max}$ are the minimum and maximum degrees of $G$ respectively.
			According to numerical experiments the analog result is true for the energy of vertices, that is:
			\[d_{min} \mathcal{LE}_G(v) \leq \mathcal{E}_G(v) \leq d_{max} \mathcal{LE}_G(v)\]
			for all $v\in V.$
		\end{remark}
	
	\subsection{Bounds with the Cheeger Constants and Ollivier-Ricci Curvature}
	
		In this section we relate the energy of vertices with two geometric quantities. This can be done because the relationship between the spectrum of the normalized Laplacian and those quantities have been studied before. According to Chung \cite{Chung97}, this relationship happens because the matrix $\mathcal{L}$ corresponds in a natural way to the Laplace-Beltrami operator for Riemannian manifolds.
		
		The first quantities of interest are the Cheeger constant and the dual Cheger constant introduced in \cite{BauerJost13}. The \textit{Cheeger constant} is interpreted as a measure of how well can the graph be partitioned into two subgraphs with few edges between them; formally it is defined as:
		\[h(G) := \min_{U\subset V} \frac{|E(U, V\setminus U)|}{\min\{vol(U), vol(V\setminus U)\}},\]
		where $U$ is nonempty, $vol(U) := \sum_{v_i\in U}d_i$ and $E(U, V\setminus U) \subset E$ is the set of edges with one vertex in $U$ and the other in $V\setminus U;$ in general, $h(G) \leq 1$. In Theorem 2.3 of \cite{Chung97} can be found a proof of the following inequality
		\begin{equation}
		\label{eq:CheegerInequality}
		\lambda_1 (\mathcal{L})> 1- \sqrt{1-h(G)^2}.
		\end{equation}
		
		A similar quantity, $\bar{h}(G)$ was introduced in \cite{BauerJost13}. It was called the \textit{dual Cheeger constant} and it is defined as
		 \[\bar{h}(G) := \max_{V_1,V_2}\frac{2|E(V_1, V_2)|}{vol(V_1)+vol(V_2)},\]
		 where $V_1,V_2,V\setminus(V_1 \cup V_2)$ form a partition of $V$. Intuitively, this constant can be interpreted as a measure of how far is the G to be bipartite. For us, the utility of $\bar{h}(G)$ lies in the following inequality contained in Theorem 3.2 in \cite{BauerJost13}:
		\begin{equation}
		\label{eq:dualCheegerInequality}
		\lambda_{n} (\mathcal{L}) \leq 1 + \sqrt{1-(1-\bar{h}(G))^2},
		\end{equation} 
		We are in position to prove the next result.
		
		\begin{theorem}
			\label{thm:vertexEnergyAndCheegerConstant}
			Let $G=(V,E)$ be a connected graph, $h(G),\, \bar{h}(G)$ its Cheeger  and dual Cheeger constants. If $1\leq i \leq n$, then
			\[\mathcal{LE}(v_i) \leq \frac{d_i}{2m} + \alpha \left(1-\frac{d_i}{2m}\right),\]
			where $\alpha=\max\bigg\{\sqrt{1-h(G)^2},\sqrt{1-(1-\bar{h}(G))^2}\bigg\}.$
		\end{theorem}
		\begin{proof}
			Combining inequalities \ref{eq:CheegerInequality} and \ref{eq:dualCheegerInequality} we arrive at
			\[ 1- \sqrt{1-h(G)^2}\leq \lambda_1(\mathcal{L}) \leq \cdots \leq \lambda_{n-1}(\mathcal{L}) \leq 1 + \sqrt{1-(1-\bar{h}(G))^2},\]
			which implies $|\lambda_j(\mathcal{L})-1| \leq \alpha $ for $j=1,\cdots,n-1$. So we can proceed as follows, using the spectral decomposition of $\mathcal{L}$ and proposition \ref{prop:EnergySpectral},
			\begin{align*}
			\mathcal{LE}(v_i) &= \sum_{j=0}^{n-1}  V_{ij}^2 |\lambda_j(\mathcal{L}) - 1|\\
			&= \frac{d_i}{2m} + \sum_{j=1}^{n-1}  V_{ij}^2 |\lambda_j(\mathcal{L}) - 1| \\
			&\leq \frac{d_i}{2m} + \alpha\sum_{j=1}^{n-1}  V_{ij}^2 \\
			&=  \frac{d_i}{2m} + \alpha \left(1-\frac{d_i}{2m}\right),
			\end{align*}
		where in the second equality we used that $D^{1/2}\mathbf{1}$ is an eigenvector of $\mathcal{L}$ with eigenvalue $0$, with $\mathbf{1}$ a vector with all its entries equal to $1$ and so we can take $V_{0i}^2 =d_i/2m$.
		\end{proof}
		
		Adding the normalized Laplacian energies of the graph we arrive at the following bound for the normalized Laplacian energy.
		
		\begin{corollary}
			Let $G=(V,E)$ be a connected graph, $h(G), \bar{h}(G)$ its Cheeger  and dual Cheeger constants, and $\alpha=\max\bigg\{\sqrt{1-h(G)^2},\sqrt{1-(1-\bar{h}(G))^2}\bigg\}.$ Then
			\[\mathcal{LE}(G)\leq 1+\alpha(n-1).\]
		\end{corollary}
		
		The second geometric quantity of our interest is the Ollivier-Ricci curvature introduced in \cite{Ollivier09}. In order to define it, we need to introduce the following family of probability measures; let $v\in V$, then
		\[
		m_v(w) := \begin{cases}
		\frac{1}{d_v} & \text{ if } w\sim v,\\
		0 & \text{ otherwise;}
		\end{cases}
		\]
		these measures can be interpreted as the transition probabilities of a random walk in the vertices of the graph. The Ollivier-Ricci curvature between two vertices $v$ and $w$ is defined as
		\[\kappa(v,w):= 1 - \frac{W_1(m_v,m_w)}{d(v,w)},\]
		where $d(v,w)$ is the minimum number of edges in a path joining $v$ and $w$ and $W_1(m_v,m_w)$ denotes the Wasserstein distance (or transportation distance) between $m_x$ and $m_y$, formally
		\[W_1(m_v,m_w): = \min_{\pi} \sum_{(x,y)\in V\times V} d(x,y)\pi(x,y),\]
		where the minimum is taken over all couplings (or transport plans) between $m_v$ and $m_w,$ that is, over all measures $\pi$ on $V\times V$ with marginals $m_v$ and $m_w$. Intuitively, $W_1(m_v,m_w)$ is the minimum cost of transporting the mass of $m_v$ to $m_y$ with the distance as the cost function. 
		
		We can give a bound on the normalized Laplacian energy of a vertex in terms of a bound on the Ollivier-Ricci Curvature $\kappa$ because of the following result:
		
		\begin{theorem}[\cite{Ollivier09}]
			\label{thm:spectralGapOllivier}
			Let a graph $G=(V,E)$ be such that $\kappa (x,y)\geq k,$ $\forall \{x,y\} \in E,$ then the eigenvalues of the normalized graph Laplace operator $\mathcal{L}$ satisfy
			\[k\leq \lambda_1(\mathcal{L}) \leq \cdots \leq \lambda_{n-1}(\mathcal{L}) \leq 2-k. \]
		\end{theorem}
		According to \cite{BauerJostLiu12}, this theorem is just Proposition 30 of \cite{Ollivier09} stated finite graphs. Using Theorem \ref{thm:spectralGapOllivier} we can prove the following:
		
		\begin{theorem}
			\label{thm:vertexEnergyAndCurvature}
			Let a graph $G=(V,E)$ be such that $\kappa (v,w)\geq k,$ $\forall \{v,w\} \in E,$ then 
			\[\mathcal{LE}(v_i) \leq 1-k\left(1-\frac{d_i}{2m}\right)\]
			for every vertex $v_i \in V.$
		\end{theorem}
		\begin{proof}
			On the one hand, Theorem \ref{thm:spectralGapOllivier} implies 
			\[|\lambda_1(\mathcal{L})-1|\leq 1-k,\; \text{ for } i=1,\cdots,n-1.\]
			On the other, proposition \ref{prop:EnergySpectral} allow us to write
			\begin{align*}
			\mathcal{LE}(v_i) &= \sum_{j=0}^{n-1}  V_{ij}^2 |\lambda_j(\mathcal{L}) - 1|\\
			&= \frac{d_i}{2m} + \sum_{j=1}^{n-1}  V_{ij}^2 |\lambda_j(\mathcal{L}) - 1| \\
			&\leq \frac{d_i}{2m} + (1-k)\sum_{j=1}^{n-1}  V_{ij}^2 \\
			&=  \frac{d_i}{2m} + (1-k) \left(1-\frac{d_i}{2m}\right) \\
			&= 1-k\left(1-\frac{d_i}{2m}\right),
			\end{align*}
			where we have used $V_{0i}^2 =d_i/2m$ for every $i$.
		\end{proof}
		We note that this last bound is only meaningful if $k\in (0,1).$ We finish this section by relating the normalized Laplacian energy of the graph with a bound on the Ollivier-Ricci curvature.
		
		\begin{corollary}
			Let a graph $G=(V,E)$ be such that $\kappa (v,w)\geq k,$ $\forall \{v,w\} \in E,$ then 
			\[\mathcal{LE}(G) \leq n-k\left(n-1\right).\]
		\end{corollary}

\section{Examples}
	\label{section:Examples}

	In this section we present formulas for the Laplacian energies of vertices for the star and the path graphs. Observe that, for regular graphs one can easily obtain the Laplacian energies of vertices using the (adjacency) energy with aid of lemmas \ref{lemma:regularGraphs} and \ref{lemma:regularGraphsNormalized}. See section 5.1 in \cite{ArizmendiEtAl18} for examples of this type.
	
	\subsection{Star Graphs}
	\label{section:StarGraphs}
	
	The star graph $S_n$ is a graph with $n$ vertices, $V=\{v_1,v_2,\cdots,v_n\}$, one vertex is connected to every other and there are no more edges, that is $E=\{\{v_1,v_i\}: i=2,\cdots,n\}.$ It is easily seen that 
	$2mn^{-1} = 2(n-1)n^{-1}$.
	
	The spectrum of the Laplacian $L(S_n),$ is $0,$ $1$ (with multiplicity $n-2$) and $n$. Its eigenvectors are $V_1 = (1,1,\cdots,1)$ (with eigenvalue $0$); for every $n-1\geq i\geq 2 $ the vectors ${V}_{i} = (\nu_1,\nu_2,\cdots, \nu_n)$
	given by 
	\[
	\nu_k =\begin{cases}
	1 & k=i\\
	-1 & k=i+1\\
	0 & otherwise,
	\end{cases}
	\] 
	are eigenvectors with eigenvalue $1$; finally, $(-(n-1), 1, 1, \cdots, 1)$ is an eigenvector with eigenvalue $n$.
	
	To compute the Laplacian energy of the vertices using lemma \ref{prop:EnergySpectral}, one needs to find an orthonormal basis of eigenvectors. Nevertheless, in this case there is no need to do that, because applying the Gramm-Schmidt process to the vectors $\{V_i\}_{i=2}^{n-1}$ would lead to vectors whose first component is $0$. By lemma \ref{prop:EnergySpectral} we have
	
	\begin{align*}
	L\mathcal{E}_{S_n}(v_1) 
	&= \frac{1}{n} \bigg|0-2 \frac{n-1}{n} \bigg| + \frac{n-1}{n} \bigg|n -2 \frac{n-1}{n} \bigg| \\
	&= \frac{n-1}{n^2}(n^2 - 2n + 4).
	\end{align*}
	
	The rest of the vertices have all the same Laplacian energy. So, by equation \ref{eq:sumOfLEnergies}, one would have 
	\begin{align*}
	L\mathcal{E}_{S_n}(v_1) + (n-1)L\mathcal{E}_{S_n}(v_2) &= L\mathcal{E}(S_n) \\
	&= \frac{2n-2}{n} +(n-2) \bigg|1-\frac{2n-2}{n}\bigg| +\bigg|n - \frac{2n-2}{n}\bigg| \\
	&=  \frac{(n-2)^2}{n}	+ n,
	\end{align*}
	which leads to
	\begin{align*}
	LE(v_2) 
	&= \frac{n^3 - n^2 - 2n + 4}{n^2(n-1)}.
	\end{align*}
	We can conclude:
	\[
	L\mathcal{E}_{S_n}(v_k) = \begin{cases}
	(n-1)(n^2 - 2n + 4)n^{-2} & \text{ if } k=1,\\
	(n^3 - n^2 - 2n + 4)n^{-2}(n-1)^{-1} & \text{ if } k \geq 2.
	\end{cases}
	\]
	Note that: 
	\[LE(v_1) \neq (n-1)LE(v_2),\]
	so, in general, Proposition 3.9 of \cite{ArizmendiEtAl18} is not true for the Laplacian energy of a vertex.	
	
	The same procedure as above give us the result for the normalized Laplacian energy of vertices, however we can use the bounds given in theorems \ref{thm:NLEUpperBound} and \ref{thm:NLELowerBound} to compute the normalized Laplacian energy of $v_1$:
	
	\[1 = \frac{1}{d_1} \sum_{w:w \sim v_1} \frac{1}{d_w} \leq \mathcal{LE}_{S_n}(v_1) \leq \sqrt{\frac{1}{d_1} \sum_{w:w \sim v_1} \frac{1}{d_w}} = 1.\]
	For $2\leq i \leq n$, knowing that the spectrum of the normalized Laplacian of $S_n$ is $0$, $1$ (with multiplicity $n-2$) and $2$:
	\begin{align*}
	\mathcal{LE}_{S_n}(v_1) + (n-1)\mathcal{LE}_{S_n}(v_2) &= \mathcal{LE}(S_n) =2,
	\end{align*}
	which leads to $\mathcal{LE}_{S_n}(v_i) = (n-1)^{-1}$ for $2\leq i \leq n.$ Finally:
	\[
	\mathcal{LE}_{S_n}(v_k) = \begin{cases}
	1 & \text{ if } k=1,\\
	(n-1)^{-1} & \text{ if } k \geq 2.
	\end{cases}
	\]
	As we see, the vertices of the Star graph attain the bounds given in corollary \ref{cor:NLEDegreeBounds}.
	\subsection{Path Graphs}
	
	Let $P_n =(V, E)$ be the path graph in $n$ vertices. That is, $V=\{1,2,\cdots, n\},$ and $E=\{(a,a+1):1\leq a < n\}.$ The mean degree of $P_n$ is
	\[ \frac{2m}{n} = \frac{1}{n} \sum_{i=1}^{n} d_i = 
	\frac{1}{n} ( 2 + 2(n-2)) = 
	\frac{1}{n} ( 2 + 2n - 4) = \frac{1}{n} ( 2n - 2) .
	\]
	
	According to Lemma 6.6.1 in \cite{Spielman2019}, the eigenvalues of $P_n$ are $2(1-\cos(\pi k/n)),$ with eigenvectors
	\[V_k(a) = \cos(\pi k a /n - \pi k/2n),\]
	for $0\leq k < n,$ we note that all the eigenvalues are distinct, because the function $x \mapsto 2(1-\cos(\pi x/n))$ is increasing in the interval $[0,n].$

	If $k=0,$ the norm of $V_k$ is $\sqrt{n},$ else
	\begin{align*}
	\|V_k \|^2 &= \sum_{a=1}^n\cos^2(\pi k a /n - \pi k/2n) 
	= \sum_{a=1}^n \frac{1}{2} \left[ \cos \left(\frac{2\pi k}{n}a - \frac{\pi k}{n}\right) + 1
	\right] \\
	&= \frac{n}{2}+ \frac{1}{2}\sum_{a=1}^n \cos \left(\frac{2\pi k}{n}a - \frac{\pi k}{n}\right)\\
	&= \frac{n}{2}+ \frac{1}{2} \left( \frac{\sin\left(-\frac{n+1}{2} \frac{2\pi k}{n} \right) \cos\left(-\frac{\pi k}{n}+2\pi k\right)}{\sin(-\pi k /n)} - \cos\left(-\frac{\pi k}{´n}\right)	\right)\\
	&= \frac{n}{2}+ \frac{1}{2} \left( \frac{\sin\left(- \frac{\pi k}{n} \right) \cos\left(-\frac{\pi k}{n}\right)}{\sin(-\pi k /n)} - \cos\left(-\frac{\pi k}{´n}\right)	\right)= \frac{n}{2}.
	\end{align*} 
	Finally, we can use lemma \ref{prop:EnergySpectral} to compute the energy of the vertices.
	\begin{align*}
	L\mathcal{E}_{P_n}(v_k) &= \sum_{i=0}^{n-1}  V_{ki}^2 |\mu_i - (2m/n)|\\
	&= \frac{1}{n} \frac{2(n-1)}{n} +\sum_{i=1}^{n-1}  \frac{2\cos^2(\pi i k /n - \pi i/2n)}{n} \bigg|2(1-\cos(\pi i /n))- \frac{2(n-1)}{n}\bigg|\\
	&= 2\frac{n-1}{n^2} +4\sum_{i=1}^{n-1}  \frac{\cos^2(\pi i k /n - \pi i/2n)}{n} \bigg|\frac{1}{n}-\cos(\pi i /n)\bigg|.\\
	\end{align*}
	
	{\color{red}
	}

\bibliographystyle{plain}
\bibliography{bibVertex} 

\end{document}